\documentclass[a4paper, 12pt]{amsart}
\usepackage{amsmath}
\usepackage{graphicx}
\usepackage{amsfonts}
\usepackage{amssymb}

\newtheorem{theorem}{Theorem}[section]
\newtheorem{proposition}[theorem]{Proposition}
\newtheorem{lemma}[theorem]{Lemma}
\newtheorem{corollary}[theorem]{Corollary}
\theoremstyle{definition}
\newtheorem{definition}[theorem]{Definition}

\begin{document}
	\date{2022-7-28}
	\title[Wasserstein metrics for von Neumann algebras]{Quadratic Wasserstein metrics for von Neumann algebras via transport plans}
	\author{Rocco Duvenhage}
	\address{Department of Physics\\
		University of Pretoria\\
		Pretoria 0002\\
		South Africa}
	\email{rocco.duvenhage@up.ac.za}

\begin{abstract}
We show how one can obtain a class of quadratic Wasserstein metrics, that is
to say, Wasserstein metrics of order 2, on the set of faithful normal states
of a von Neumann algebra $A$, via transport plans, rather than through a
dynamical approach. Two key points to make this work, are a suitable
formulation of the cost of transport arising from Tomita-Takesaki theory and
relative tensor products of bimodules (or correspondences in the sense of
Connes). The triangle inequality, symmetry and $W_{2}(\mu,\mu)=0$ all work
quite generally, but to show that $W_{2}(\mu,\nu)=0$ implies $\mu=\nu$, we
need to assume that $A$ is finitely generated.

\end{abstract}

\maketitle

\section{Introduction}

In recent years, several papers have studied Wasserstein metrics of order 2, or
quadratic Wasserstein metrics, in a noncommutative context using a dynamical
approach inspired by the dynamical formulation of Benamou and Brenier
\cite{BB} in the classical case. These papers, often involving finite
dimensional algebras, include \cite{CM1, CM2, CM3, H1, H2, CGT, CGGT, W}. The
papers \cite{H1, W} investigate this approach for W*-algebras, i.e., von
Neumann algebras, motivated by, among other goals, finding a noncommutative
analogue of Ricci curvature bounds. In turn, the papers \cite{BGJ1, BGJ2} made
use of these developments to obtain logarithmic Sobolev inequalities for
quantum Markov semigroups

Classically, however, there is also another approach, using transport plans
(i.e., couplings of probability measures). This approach has also been followed
in the literature on the noncommutative case, but it typically seems to lead
to a Wasserstein ``distance'' which is not zero between a state and itself,
while the triangle equality also has to be adapted. In other words, so far in
the literature on the noncommutative transport plan approach, the Wasserstein
distance on the state space of the algebra concerned, is not actually a
metric. Quadratic Wasserstein distance functions obtained from transport plans
have nevertheless already had applications in the noncommutative case,
specifically in quantum physics, despite not being metrics. See in particular
\cite{GMP, GP1, GP2, CGP, dePT}. This motivates a further study of the
transport plan formulation, in particular finding an approach where they do
deliver metrics, which is exactly our aim in this paper.

In \cite{AF}, difficulties in the noncommutative transport plan approach were
pointed out. They obtained simple Wasserstein-like metrics for the states on
the $2\times2$ matrices using noncommutative transport plans, but by deviating
somewhat from the standard conceptual framework of optimal transport.

In the case of the space of tracial states on a C*-algebra, \cite{BV} did
obtain Wasserstein metrics (of any order) via a coupling approach, through a
somewhat different formulation of couplings (transport plans) than in this
paper. Furthermore, our setup involves bimodules for von Neumann algebras in a central way, which is again different from the approach in \cite{BV}. 
Nevertheless, there are some analogies between the two approaches, which will be pointed out where relevant.

We can also mention the well-known work of Connes \cite{Con89} on
noncommutative geometry, where a metric on the state space is obtained as a
Wasserstein distance of order 1. However, this is not based on couplings, but
rather extends the classical dual definition based on Lipschitz seminorms.

There has been at least one other approach to a quantum Wasserstein distance,
also of order 1, namely \cite{DMTL}, based on the notion of neighbouring
states, where a metric was obtained.

In this paper we develop an approach to noncommutative quadratic Wasserstein
distances using transport plans, where the distance functions on the space of
faithful normal states of a von Neumann algebra are indeed metrics. Below we
outline our approach, starting with brief remarks on the classical case.

We are interested in obtaining a non-commutative analogue of the classical
case where the cost function is taken as the usual distance squared in
$\mathbb{R}^{n}$, giving the following cost of transport from one probability
measure to another on some closed subset $X$ of $\mathbb{R}^{n}$:
\begin{equation}
I(\omega)=\int_{X\times X}\left\|  x-y\right\|  ^{2}d\omega
(x,y)\label{klasKoste}
\end{equation}
where $\omega$ is a transport plan from the one probability measure $\mu$ to
the other $\nu$, i.e. $\omega$ is a coupling of the two measures. The goal is
to find the minimum of $I(\omega)$, for a given pair $(\mu,\nu)$, the square
root of which then defines the quadratic Wasserstein distance between $\mu$
and $\nu$, giving a metric on an appropriate set of probability measures. This
is a specific but important case of the optimal transport problem. A clear
discussion of the original motivation for optimal transport, as well as its
modern implications, can be found in the books by Villani \cite{V1, V2}.

As we only consider the quadratic case in this paper, we henceforth simply
refer to Wasserstein distances (should they not yet be shown to be metrics)
and Wasserstein metrics.

Our noncommutative version will correspond to the case where $X$ is a bounded
set, ensuring that the coordinate functions $x\mapsto x_{l}$, in terms of
$x=(x_{1},...,x_{n})$, are bounded. We replace these coordinate functions by
elements $k_{1},...,k_{n}$ of a von Neumann algebra $A$. The $k_{l}$ are then
bounded operators, which is why this corresponds to bounded $X$. In this
approach the closest analogy to the classical coordinate functions would be to
assume that the $k_{l}$ are self-adjoint, but for our development this is not
actually needed, though we do ultimately need a weaker assumption, as will be
pointed out below.

A direct translation of the classical cost function $c(x,y):=\left\|
x-y\right\|  ^{2}$ above, is then
\[
\sum_{l=1}^{n}\left|  k_{l}\otimes1-1\otimes k_{l}\right|  ^{2}
\]
where $1$ denotes the unit of $A$ and $\left|  a\right|  ^{2}:=a^{\ast}a$ for
$a$ in any $\ast$-algebra. However, noncommutativity of $A$ makes this direct
translation untenable.

Instead, for the cost of transport between two faithful normal states $\mu$
and $\nu$ on $A$, we, at least heuristically, propose to use
\begin{equation}
c=\sum_{l=1}^{n}\left|  k_{l}\otimes1-1\otimes(S_{\nu}k_{l}^{\ast}S_{\nu
})\right|  ^{2}\label{heuristieseKoste}
\end{equation}
where $S_{\nu}$ is the conjugate linear (and typically unbounded) linear
operator from Tomita-Takesaki theory associated with $\nu$. As this $c$
appears to depend on one of the states, namely $\nu$, it is immediately clear
that this is quite different from the classical case, where no such dependence
on the probability measures are present. However, even though $c$ has been
``contextualized'' in this sense for the states involved, the elements $k_{l}$
of $A$ do not depend on the states. Furthermore, it is not clear that this can
lead to a distance that is indeed symmetric, but the apparent asymmetry
between the two states $\mu$ and $\nu$ in $c$ can be circumvented using a
condition involving the modular group. The first indication that we should use
(\ref{heuristieseKoste}), was that it leads to zero distance between a state
and itself (see Proposition \ref{W2(mu,mu)=0}).

The fact that $S_{\nu}$ is typically unbounded and $S_{\nu}k_{l}^{\ast}S_{\nu
}$ not an element of $A$ (in standard form) nor of its commutant $A^{\prime}$,
is a complication which we can sidestep by rewriting the cost for a
noncommutative transport plan in a more convenient form which will give us a
rigorous definition of the cost of transportation form $\mu$ to $\nu$. We then
do not refer to $c$ above directly, but only to the cost of transport from one
state to another in terms of the $k_{l}$. A detailed discussion of this
appears in Section \ref{AfdKoste&Afst}.

This approach turns out to fit perfectly with bimodules, or correspondences in
the sense of Connes, and the relative tensor product of bimodules provides the
ideal structure to prove the triangle inequality for the Wasserstein distance
$W_{2}$ which we are going to define along the lines above.

An important point is the following: To complete the proof that our $W_{2}$ is
a metric, we assume that $A$ is finitely generated, in analogy to
$\mathbb{R}^{n}$ being finite dimensional. More precisely, that $A$ is
generated by $k_{1},...,k_{n}$ above, with $\{k_{1}^{\ast},...,k_{n}^{\ast
}\}=\{k_{1},...,k_{n}\}$; note that $k_{l}^{\ast}=k_{l}$ is not required. This
is needed to show that $W_{2}(\mu,\nu)=0$ for faithful normal states on $A$,
implies that $\mu=\nu$. However, the triangle inequality, symmetry and
$W_{2}(\mu,\mu)=0$ all work without this assumption.

We note that our use of generators parallels the approach of \cite{BV}. In
particular, they also assume their C*-algebra to be finitely generated to
obtain Wasserstein metrics on the space of tracial states.

Transport plans and related cyclic representations are discussed in Section
\ref{AfdOordPlan}, setting up the basic framework and much notation for the
rest of the paper. Our basic definitions related to the cost of transport and
Wasserstein distances follow in Section \ref{AfdKoste&Afst}. The subsequent
three sections then prove the various properties of a metric in turn. After
that, the last section gives an indication of how the setting we use may be expanded.

\section{Transport plans and cyclic representations}\label{AfdOordPlan}

This section sets up the framework to be used in the rest of the paper, in
particular the formulation of transport plans. Transport plans lead to cyclic
representation which play a central role in this paper. These representations
and the bimodule structure on their Hilbert spaces, are also discussed in this
section. The notation introduced here will be used throughout the sequel.

Although our focus is eventually on the faithful normal states of one von
Neumann algebra at a time, a number of aspects of our approach are clarified
by starting off with three von Neumann algebra. This is ultimately related to
the three elements of a metric space appearing in the formulation of the
triangle inequality.

We consider three (necessarily $\sigma$-finite) von Neumann algebras $A$, $B $
and $C$ with faithful normal states $\mu$, $\nu$ and $\xi$ respectively. We
assume that they are all in standard form, meaning that $A$ is a\ von Neumann
algebra on a Hilbert space $G_{\mu}$ with a cyclic and separating vector
$\Lambda_{\mu}$ for $A$ such that
\[
\mu(a)=\left\langle \Lambda_{\mu},a\Lambda_{\mu}\right\rangle
\]
for all $a\in A$, which allows us to define a state $\mu^{\prime}$ on the
commutant $A^{\prime}$ by
\[
\mu^{\prime}(a^{\prime})=\left\langle \Lambda_{\mu},a^{\prime}\Lambda_{\mu
}\right\rangle
\]
for all $a^{\prime}\in A^{\prime}$. The corresponding notation will be used
for $(B,\nu)$ and $(C,\xi)$ as well.

The starting point of our development is the definition of \emph{transport
plans}, also known as \emph{couplings}.

\begin{definition}
\label{oordPlan}A \emph{transport plan from }$\mu$\emph{\ to }$\nu$ is a state
$\omega$ on the algebraic tensor product $A\odot B^{\prime}$ such that
\[
\omega(a\otimes1)=\mu(a)\text{ \ \ and \ \ }\omega(1\otimes b^{\prime}
)=\nu^{\prime}(b^{\prime})
\]
for all $a\in A$ and $b^{\prime}\in B^{\prime}$. Denote the set of all
transport plans from $\mu$ to $\nu$ by $T(\mu,\nu)$.
\end{definition}

A transport plan from $\nu$ to $\xi$ will be denoted by $\psi$, and from $\mu$
to $\xi$ by $\varphi$.

To obtain symmetry of a Wasserstein distance on the faithful normal states, we
are going to use a restricted set of transport plans. To define them, we use a
property which was called balance between dynamical systems in \cite{DS2}:

\begin{definition}
\label{balans}Given \emph{dynamical systems} $\mathbf{A}=(A,\alpha,\mu)$ and
$\mathbf{B}=\left(  B,\beta,\nu\right)  $, where $\alpha$ and $\beta$ are
unital completely positive (u.c.p.) maps $\alpha:A\rightarrow A$ and
$\beta:B\rightarrow B$ such that $\mu\circ\alpha=\mu$ and $\nu\circ\beta=\nu$,
we say that $\mathbf{A}$ and $\mathbf{B}$ (in this order) are in
\emph{balance} with respect to a transport plan $\omega$ from $\mu$ to $\nu$,
written as
\[
\mathbf{A}\omega\mathbf{B},
\]
if
\[
\omega(\alpha(a)\otimes b^{\prime})=\omega(a\otimes\beta^{\prime}(b^{\prime}))
\]
for all $a\in A$ and $b^{\prime}\in B^{\prime}$, where $\beta^{\prime
}:B^{\prime}\rightarrow B^{\prime}$ is the dual of $\beta$ defined by
\[
\left\langle \Lambda_{\nu},b\beta^{\prime}(b^{\prime})\Lambda_{\nu
}\right\rangle =\left\langle \Lambda_{\nu},\beta(b)b^{\prime}\Lambda_{\nu
}\right\rangle
\]
for all $b\in B$ and $b^{\prime}=B^{\prime}$ (see \cite{AC} for the theory
behind such duals and \cite[Section 2]{DS2} for a summary).
\end{definition}

Our restricted set of transport plans are then defined as follows:

\begin{definition}
\label{modOord}The set of \emph{modular} transport plans from $\mu$ to $\nu$
is
\[
T_{\sigma}(\mu,\nu):=\left\{  \omega\in T(\mu,\nu):(A,\sigma_{t}^{\mu}
,\mu)\omega\left(  B,\sigma_{t}^{\nu},\nu\right)  \text{ for all }
t\in\mathbb{R}\right\}
\]
where $\sigma^{\mu}$ and $\sigma^{\nu}$ are the modular groups associated with
$\mu$ and $\nu$ respectively.
\end{definition}

Note that $T_{\sigma}(\mu,\nu)$ is not empty, as it contains $\mu\odot
\nu^{\prime}$.

In the remainder of this section we consider general transport plans, but from
the next section modular transport plans will become the focus.

Transport plans lead to cyclic representations and we need to keep track how
the various cyclic representations relate to one another. Let $(H_{\omega}
,\pi_{\omega},\Omega)$ be a cyclic representation of $(A\odot B^{\prime
},\omega)$, i.e. $\Omega\in H_{\omega}$ is cyclic for $A\odot B^{\prime}$ and
\[
\omega=\left\langle \Omega,\pi_{\omega}(\cdot)\Omega\right\rangle _{\omega},
\]
where for emphasis we write the inner product of $H_{\omega}$ as $\left\langle
\cdot,\cdot\right\rangle _{\omega}$. This induces cyclic representations
$(H_{\mu}^{\omega},\pi_{\mu}^{\omega},\Omega)$ and $(H_{\nu}^{\omega},\pi
_{\nu^{\prime}}^{\omega},\Omega)$ of $(A,\mu)$ and $(B^{\prime},\nu^{\prime})$
respectively, by setting
\[
H_{\mu}^{\omega}:=\overline{\pi_{\omega}(A\otimes1)\Omega}\text{\ \ \ and
\ }\pi_{\mu}^{\omega}(a):=\pi_{\omega}(a\otimes1)|_{H_{\mu}^{\omega}}
\]
for all $a\in A$, and
\[
H_{\nu}^{\omega}:=\overline{\pi_{\omega}(1\otimes B^{\prime})\Omega}\text{
\ \ and \ }\pi_{\nu^{\prime}}^{\omega}(b^{\prime}):=\pi_{\omega}(1\otimes
b^{\prime})|_{H_{\nu}^{\omega}}
\]
for all $b^{\prime}\in B^{\prime}$. A unitary equivalence
\begin{equation}
u_{\nu}:G_{\nu}\rightarrow H_{\nu}^{\omega}\label{GnuHnu}
\end{equation}
from $(G_{\nu},$id$_{B^{\prime}},\Lambda_{\nu})$ to $(H_{\nu}^{\omega}
,\pi_{\nu^{\prime}},\Omega)$ is given by
\[
u_{\nu}b^{\prime}\Lambda_{\nu}:=\pi_{\nu^{\prime}}^{\omega}(b^{\prime})\Omega
\]
for all $b^{\prime}\in B^{\prime}$. Then
\[
\pi_{\nu^{\prime}}^{\omega}(b^{\prime})=u_{\nu}b^{\prime}u_{\nu}^{\ast}
\]
for all $b^{\prime}\in B^{\prime}$. By setting
\[
\pi_{\nu}^{\omega}(b):=u_{\nu}bu_{\nu}^{\ast}
\]
for all $b\in B$, we also obtain a cyclic representation $(H_{\nu}^{\omega
},\pi_{\nu}^{\omega},\Omega)$ of $(B,\nu)$.

Denoting the modular conjugation for $B$ associated to $\Lambda_{\nu}$ by
$J_{\nu}$, while the modular conjugation for $B$ associated to $\Omega$ is
denoted by $J_{\nu}^{\omega}$, one finds that
\[
J_{\nu}^{\omega}u_{\nu}=u_{\nu}J_{\nu}
\]
from which it follows that
\[
\pi_{\nu}^{\omega}=j_{\nu}^{\omega}\circ\pi_{\nu^{\prime}}^{\omega}\circ
j_{\nu}%
\]
where $j_{\nu}^{\omega}:=J_{\nu}^{\omega}(\cdot)^{\ast}J_{\nu}^{\omega}$ and
$j_{\nu}:=J_{\nu}(\cdot)^{\ast}J_{\nu}$ on $\mathcal{B}(H_{\nu}^{\omega})$ and
$\mathcal{B}(G_{\nu})$ respectively. Here $\mathcal{B}(G_{\nu})$ is the von
Neumann algebra of all bounded linear operators $G_{\nu}\rightarrow G_{\nu}$.

A key structure that emerges from the transport plan $\omega$, is that
$H_{\omega}$ is an $A$-$B$-bimodule via
\[
axb:=\pi_{\omega}(a\otimes j_{\nu}(b))x
\]
for all $a\in A$, $b\in B$ and $x\in H_{\omega}$. The required normality
conditions follow from \cite[Theorem 3.3]{BCM}.

The forgoing discussion of course also holds for $\psi$ and $\varphi$, using
corresponding notation for the induced representations, and in particular
denoting the cyclic vectors by
\[
\Psi\text{ and }\Phi
\]
respectively.

A central fact which will be of great importance to us, is that transport
plans from $\mu$ to $\nu$ are in a one-to-one correspondence with u.c.p. maps
$E:A\rightarrow B$ such that $\nu\circ E=\mu$. Let
\[
\varpi_{B}:B\odot B^{\prime}\rightarrow\mathcal{B}(G_{\nu})
\]
be the unital $\ast$-homomorphism defined by extending $\varpi_{B}(b\otimes
b^{\prime})=bb^{\prime}$ using the universal property of tensor products. Then
set
\[
\delta_{\nu}=\left\langle \Lambda_{\nu},\varpi_{B}(\cdot)\Lambda_{\nu
}\right\rangle .
\]
Note that $\delta_{\nu}$ is a transport plan from $\nu$ to itself. It now
follows that there is a unique map
\[
E_{\omega}:A\rightarrow B
\]
such that%
\begin{equation}
\omega(a\otimes b^{\prime})=\delta_{\nu}(E_{\omega}(a)\otimes b^{\prime
})\label{E}%
\end{equation}
for all $a\in A$ and $b^{\prime}\in B^{\prime}$. This map $E_{\omega}$ is
linear, normal (i.e. $\sigma$-weakly continuous), u.c.p. (unital and
completely positive) and satisfies
\[
\nu\circ E_{\omega}=\mu.
\]
Conversely, given a u.c.p. map $E:A\rightarrow B$ such that $\nu\circ E=\mu$,
it defines a transport plan $\omega_{E}$ from $\mu$ to $\nu$ by
\[
\omega_{E}(a\otimes b^{\prime})=\delta_{\nu}(E(a)\otimes b^{\prime})
\]
which satisfies $E=E_{\omega_{E}}$. Technical details can be found in
\cite[Section 3]{DS2}. This correspondence is well-known and widely used in
various forms and degrees of generality, for example in finite dimensions in
quantum information theory, where it is known as the Choi-Jamio{\l}kowski duality
(see \cite{Choi, deP, Jam} for the origins and \cite{JLF} for an overview),
and has also found application in the theory of noncommutative joinings
\cite{BCM}. In effect we can now view u.c.p. maps as transport plans, as was
done in \cite{dePT}, however, we reserve the term ``transport plan'' for how
it is used in Definition \ref{oordPlan}. In classical
probability theory the correspondence between couplings and Markov operators
was studied for couplings of a measure with itself in \cite{B}, and more
generally in \cite{MT}, though it seems not to be widely used in classical
optimal transport.

By \cite[Theorem 4.1]{DS2} we can then express $\mathbf{A}\omega\mathbf{B}$ in
Definition \ref{balans} as
\[
E_{\omega}\circ\alpha=\beta\circ E_{\omega}%
\]
which is often a convenient way to check if a transport plan is modular.

As in \cite[Section 5]{DS2}, this correspondence allows us to compose
transport plans $\omega\in T(\mu,\nu)$ and $\psi\in T(\nu,\xi)$ to obtain a
transport plan $\omega\circ\psi\in T(\mu,\xi)$ defined via
\[
E_{\omega\circ\psi}=E_{\psi}\circ E_{\omega}.
\]
In the remainder of this section we are particularly interested in this
transport plan, and how it relates to the relative tensor product of the
bimodules $H_{\omega}$ and $H_{\psi}$. So we set
\[
\varphi=\omega\circ\psi
\]
and consider the relative tensor product
\[
H:=H_{\omega}\otimes_{\nu}H_{\psi}.
\]
An extended discussion of these products can be found in \cite[Section
IX.3]{T2}, but also see \cite{Fal} and the early work \cite{Sa}. The original
source is Connes' work on correspondences, which has not been published in
full, but see \cite[Appendix V.B]{Con94} for a partial exposition.

The relative tensor product $H$ is itself an $A$-$C$-bimodule, and contains
the vector
\[
\Phi:=\Omega\otimes_{\nu}\Psi
\]
which can indeed be taken as the cyclic vector in the cyclic representation
$(H_{\varphi},\pi_{\varphi},\Phi)$ of $(A\odot C^{\prime},\varphi)$. This is
because the representation can be obtained from the bimodule $H$ using the not
necessarily cyclic representation $\pi$ of $A\odot C^{\prime}$ defined
through
\[
\pi(a\otimes c^{\prime})z:=azj_{\xi}(c^{\prime})
\]
for all $z\in H$, by restriction:
\[
H_{\varphi}:=\overline{\pi(A\odot C^{\prime})\Phi}\text{ \ \ and \ \ }
\pi_{\varphi}(t):=\pi(t)|_{H_{\varphi}}%
\]
for all $t\in A\odot C^{\prime}$. That this indeed provides a cyclic
representation for $\varphi$, follows from \cite[Corollary 5.7]{DS2}.

\section{Cost of transport and the distance $W_{2}$}\label{AfdKoste&Afst}

Here we introduce the cost of transport as will be used in this paper,
followed by the paper's central definition, namely that of Wasserstein
distance $W_{2}$.

An important formula related to cost of transport, expressed in terms of the
setup and notation from the previous section and norm $\left\|
\cdot\right\|  _{\omega}$ of the Hilbert space $H_{\omega}$, is the following:
\begin{equation}
\left\|  \pi_{\mu}^{\omega}(a)\Omega-\pi_{\nu}^{\omega}(b)\Omega\right\|
_{\omega}^{2}=\mu(a^{\ast}a)+\nu(b^{\ast}b)-\nu(E_{\omega}(a)^{\ast}
b)-\nu(b^{\ast}E_{\omega}(a))\label{normVsE}%
\end{equation}
for all $a\in A$ and $b\in B$. This is derived by a straightforward sequence
of manipulations, in particular making use of the projection $P_{\nu}$ of
$H_{\omega}$ onto $H_{\nu}^{\omega}$, in terms of which one has
\begin{align*}
\left\langle \pi_{\mu}^{\omega}(a)\Omega,\pi_{\nu}^{\omega}(b)\Omega
\right\rangle _{\omega}  &  =\left\langle \pi_{\omega}(a\otimes1)\Omega
,P_{\nu}\pi_{\nu}^{\omega}(b)\Omega\right\rangle _{\omega}\\
&  =\left\langle u_{\nu}^{\ast}P_{\nu}\pi_{\omega}(a\otimes1)u_{\nu}
\Lambda_{\nu},b\Lambda_{\nu}\right\rangle \\
&  =\left\langle E_{\omega}(a)\Lambda_{\nu},b\Lambda_{\nu}\right\rangle \\
&  =\nu(E_{\omega}(a)^{\ast}b)
\end{align*}
by \cite[Proposition 3.1]{DS2}.

Now set
\[
A=B=C.
\]
By the theory of standard forms \cite{Ar74, Con74, Ha75} (also see
\cite[Theorem 2.5.31]{BR1}) we can consequently take $G=G_{\mu}=G_{\nu}
=G_{\xi}$ and represent the faithful normal states $\mu$, $\nu$ and $\xi$ on
$A$ by three cyclic and separating vectors $\Lambda_{\mu},\Lambda_{\nu
},\Lambda_{\xi}\in G$ for $A$. As usual in Tomita-Takesaki theory, we define a
closed conjugate linear operator $S_{\nu}$ in $G$ through
\[
S_{\nu}a\Lambda_{\nu}=a^{\ast}\Lambda_{\nu}
\]
for all $a\in A$. Following the idea (\ref{heuristieseKoste}) in the
introduction, we use the form of the cyclic representation to formally
manipulate the expression
\begin{equation}
\omega\left(  \left|  a\otimes1-1\otimes(S_{\nu}b^{\ast}S_{\nu})\right|
^{2}\right) \label{kosItvS}
\end{equation}
for a transport plan $\omega$ from $\mu$ to $\nu$ and $a,b\in A$ as if
$S_{\nu}b^{\ast}S_{\nu}$ lies in $A^{\prime}$ (it is actually only affiliated
with $A^{\prime}$), to obtain
\begin{equation}
\left\|  \pi_{\mu}^{\omega}(a)\Omega-\pi_{\nu}^{\omega}(b)\Omega\right\|
_{\omega}^{2}.\label{kosItvVoorst}
\end{equation}
Alternatively, one can obtain the right hand side of (\ref{normVsE}) from
(\ref{kosItvS}) by a different sequence of formal manipulations using
(\ref{E}).
The expression (\ref{kosItvVoorst}) will be used as the basis for rigorously
defining the cost of transport, as each term in the heuristic expression
$\omega(c)$ for the cost of transport is of the form (\ref{kosItvS}). Note
that $\omega(c)$ is a noncommutative translation of the integral in
(\ref{klasKoste}).

\begin{definition}
\label{I}Given $k_{1},...,k_{n}\in A$ and writing $k=(k_{1},...,k_{n})$, the
associated \emph{transport cost function} $I$, which gives the cost of
transport $I(\omega)$ from $\mu$ to $\nu$ for the transport plan $\omega\in
T(\mu,\nu)$ and faithful normal states $\mu$ and $\nu$ on $A$, is defined to
be
\[
I(\omega)=\left\|  \pi_{\mu}^{\omega}(k)\Omega-\pi_{\nu}^{\omega}
(k)\Omega\right\|  _{\oplus\omega}^{2}
\]
where we have written%
\[
\pi_{\mu}^{\omega}(k)\Omega\equiv\left(  \pi_{\mu}^{\omega}(k_{1}
)\Omega,...,\pi_{\mu}^{\omega}(k_{n})\Omega\right)  \in\bigoplus_{l=1}
^{n}H_{\omega}
\]
and $\left\|  \cdot\right\|  _{\oplus\omega}$ denotes the norm on
$\bigoplus_{l=1}^{n}H_{\omega}$. I.e.,
\begin{equation}
I(\omega)=\sum_{l=1}^{n}\left\|  \pi_{\mu}^{\omega}(k_{l})\Omega-\pi_{\nu
}^{\omega}(k_{l})\Omega\right\|  _{\omega}^{2}.\label{Iterme}
\end{equation}
\end{definition}

This realizes the idea outlined in the introduction, and parallels the use of
generators in \cite{BV}'s definition of Wasserstein metrics. It can be written in a second equivalent form in
terms of the u.c.p. map $E_{\omega}$ corresponding to $\omega$, using
(\ref{normVsE}), which is more convenient for certain purposes:
\begin{equation}
I(\omega)=\sum_{l=1}^{n}\left[  \mu(k_{l}^{\ast}k_{l})+\nu(k_{l}^{\ast}
k_{l})-\nu(E_{\omega}(k_{l})^{\ast}k_{l})-\nu(k_{l}^{\ast}E_{\omega}
(k_{l}))\right]  .\label{Ialt}
\end{equation}

A more complete notation for $I$ would be $I_{k}$, but no confusion will arise.

We are now in a position to define the distance function $W_{2}$ on the
faithful normal states, which we aim to prove is a metric. By a \emph{distance
function} we simply mean a function $d:X\times X\rightarrow\mathbb{R}$ on some
set $X$, which is potentially a metric on $X$, in order to distinguish it from
a verified metric.

If we directly generalize the classical case, one would define the distance
function as the infimum of the square root of the cost over all transport
plans $\omega$ from $\mu$ to $\nu$. However, in our noncommutative setup,
this causes problems with symmetry of the distance function. We solve this
problem by only allowing the modular transport plans $T_{\sigma}(\mu,\nu)$
from Definition \ref{modOord}. This is a natural restriction on the allowed
transport plans which becomes trivial when $\mu$ and $\nu$ are traces. Similar
restrictions have been used in a related context in the theory of
noncommutative joinings, also involving couplings (not viewed as transport
plans in that case). See in particular \cite{BCM, D3}. Modular transport plans
are exactly what we need to obtain symmetry in a natural way, as will be seen
in the next section.

Thus we present the main definition of this paper:

\begin{definition}
\label{W2}Denote the set of faithful normal states on a $\sigma$-finite von
Neumann algebra $A$ by $\mathfrak{F}(A)$. Given $k_{1},...,k_{n}\in A$, we
define the associated \emph{Wasserstein distance} $W_{2}$\ on $\mathfrak{F}
(A)$ by
\[
W_{2}(\mu,\nu):=\inf_{\omega\in T_{\sigma}(\mu,\nu)}I(\omega)^{1/2}
\]
for all $\mu,\nu\in\mathfrak{F}(A)$, in terms of Definition \ref{I}.
\end{definition}

More completely $W_{2}$ can be called a Wasserstein distance of order 2, or a
quadratic Wasserstein distance.

Having defined our candidate $W_{2}$, we have to investigate if it is indeed a
metric. This is done in the next three sections.

\section{The triangle inequality}

In this section we intend to show that $W_{2}$ satisfies the triangle
inequality. To do this we use the bimodule structure and relative tensor
product discussed in Section \ref{AfdOordPlan}.
This is analogous to \cite{BV}'s use of free products with amalgamation to prove the triangle inequality in their approach.

The main technical point in
this regard is the following lemma, which is a key (and we presume well-known)
feature of the quotient construction of the relative tensor product,
essentially telling us that the imbeddings of $H_{\nu}$ into $H_{\omega}$ and
$H_{\psi}$ respectively, are identified in the relative tensor product
$H=H_{\omega}\otimes_{\nu}H_{\psi}$. This is precisely the reason that the
relative tensor product is the natural setting to prove the triangle
inequality. Please refer to \cite[Section IX.3]{T2} for further background on
relative tensor products and detail on related operators used in the lemma's proof.

\begin{lemma}
\label{inbed}In terms of the setup and notation from Section
\ref{AfdOordPlan} and the imbeddings
\[
\iota_{\omega}:H_{\omega}\rightarrow H:x\mapsto x\otimes_{\nu}\Psi
\]
and
\[
\iota_{\psi}:H_{\psi}\rightarrow H:y\mapsto\Omega\otimes_{\nu}y
\]
%
%
%
one has
\[
\iota_{\omega}(\pi_{\nu}^{\omega}(b)\Omega)=\iota_{\psi}(\pi_{\nu}^{\psi
}(b)\Psi)
\]
for all $b\in B$.
\end{lemma}

\begin{proof}
Given any $b\in B$, write $x=\pi_{\nu}^{\omega}(b)\Omega$ and $y=\pi_{\nu
}^{\psi}(b)\Psi$. In setting up the relative tensor product, one considers an
operator $L_{\nu}(x):G_{\nu}\rightarrow H_{\omega}$ which in our setup can be
written as $L_{\nu}(x)=\pi_{\nu}^{\omega}(b)u_{\nu}$ in terms of
(\ref{GnuHnu}). From the bilinear form used in the quotient construction of
the relative tensor product (see \cite[Proposition IX.3.15 and Definition
IX.3.16]{T2}) one finds
\begin{align*}
\left\|  \iota_{\omega}(x)-\iota_{\psi}(y)\right\|  ^{2} &  =\left\langle
\Psi,L_{\nu}(x)^{\ast}L_{\nu}(x)\Psi\right\rangle _{\psi}-\left\langle
y,L_{\nu}(\Omega)^{\ast}L_{\nu}(x)\Psi\right\rangle _{\psi}\\
&  -\left\langle \Psi,L_{\nu}(x)^{\ast}L_{\nu}(\Omega)\Psi\right\rangle
_{\psi}+\left\langle y,L_{\nu}(\Omega)^{\ast}L_{\nu}(\Omega)\Psi\right\rangle
_{\psi}\\
&  =0,
\end{align*}
since one has $L_{\nu}(x)^{\ast}L_{\nu}(x)=u_{\nu}^{\ast}\pi_{\nu}^{\omega
}(b^{\ast}b)u_{\nu}=b^{\ast}b$ etc., and where we are using the left
$B$-module structure of $H_{\psi}$. (Here $\left\langle \cdot,\cdot
\right\rangle _{\psi}$ is the inner product of $H_{\psi}$.)
\end{proof}

The second fact we need is the following:

\begin{lemma}
\label{inbed2}In terms of the setup and notation from Section
\ref{AfdOordPlan}, with $\varphi=\omega\circ\psi$, we have
\[
\iota_{\omega}(\pi_{\mu}^{\omega}(a)\Omega)=\pi_{\mu}^{\varphi}(a)\Phi
\]
and%
\[
\iota_{\psi}(\pi_{\xi}^{\psi}(c)\Psi)=\pi_{\xi}^{\varphi}(c)\Phi
\]
for all $a\in A$ and $c\in C$.
\end{lemma}

\begin{proof}
The first is easy to verify, while the second follows by first showing that%
\[
J_{\xi}^{\varphi}(\Omega\otimes_{\nu}y)=\Omega\otimes_{\nu}J_{\xi}^{\psi}y
\]
for all $y\in H_{\xi}^{\psi}$.
\end{proof}

The triangle inequality for $W_{2}$ given in Definition \ref{W2} can now be proven.

\begin{proposition}
\label{driehoek}Consider any $k_{1},...,k_{n}\in A$, and let $W_{2}$ be the
associated Wasserstein distance on $\mathfrak{F}(A)$. Then
\[
W_{2}(\mu,\xi)\leq W_{2}(\mu,\nu)+W_{2}(\nu,\xi)
\]
for all $\mu,\nu,\xi\in\mathfrak{F}(A)$.
\end{proposition}

\begin{proof}
For $\omega\in T_{\sigma}(\mu,\nu)$ and $\psi\in T_{\sigma}(\nu,\xi)$ in terms
of Definition \ref{modOord}, we set $\varphi=\omega\circ\psi$. Note that
$\varphi\in T_{\sigma}(\mu,\xi)$, since
\[
E_{\varphi}\circ\sigma^{\mu}_t
=E_{\psi}\circ E_{\omega}\circ\sigma^{\mu}_t
=E_{\psi}\circ\sigma^{\nu}_t\circ E_{\omega}
=\sigma^{\xi}_t\circ E_{\psi}\circ
E_{\omega}
=\sigma^{\xi}_t\circ E_{\varphi},
\]
hence $(A,\sigma^{\mu},\mu)\varphi\left(  A,\sigma^{\xi},\xi\right)  $ by
\cite[Theorem 4.1]{DS2}. According to Definition \ref{I}, and applying
$\iota_{\omega}$ and $\iota_{\psi}$ componentwise to elements of direct sums,
we have
\begin{align*}
I(\varphi)^{1/2} &  =\left\|  \pi_{\mu}^{\varphi}(k)\Phi-\pi_{\xi}^{\varphi
}(k)\Phi\right\|  _{\oplus\varphi}\\
&  =\left\|  \iota_{\omega}(\pi_{\mu}^{\omega}(k)\Omega)-\iota_{\psi}(\pi
_{\xi}^{\psi}(k)\Psi)\right\|  _{\oplus H}\\
&  \leq\left\|  \iota_{\omega}(\pi_{\mu}^{\omega}(k)\Omega)-\iota_{\omega}
(\pi_{\nu}^{\omega}(k)\Omega)\right\|  _{\oplus H}\\
&  +\left\|  \iota_{\omega}(\pi_{\nu}^{\omega}(k)\Omega)-\iota_{\psi}(\pi
_{\nu}^{\psi}(k)\Psi)\right\|  _{\oplus H}\\
&  +\left\|  \iota_{\psi}(\pi_{\nu}^{\psi}(k)\Psi)-\iota_{\psi}(\pi_{\xi
}^{\psi}(k)\Psi)\right\|  _{\oplus H}\\
&  =\left\|  \pi_{\mu}^{\omega}(k)\Omega - \pi_{\nu}^{\omega}(k)\Omega\right\|
_{\oplus\omega}+\left\|  \pi_{\nu}^{\psi}(k)\Psi-\pi_{\xi}^{\psi}
(k)\Psi\right\|  _{\oplus\psi}\\
&  =I(\omega)^{1/2}+I(\psi)^{1/2}
\end{align*}
where we employed the triangle inequality in 
$\left(\bigoplus_{l=1}^{n}H,\left\|\cdot\right\|_{\oplus H}\right)$, applied Lemmas \ref{inbed} and \ref{inbed2}, and used the fact that
$\iota_{\omega}$ and $\iota_{\psi}$ preserve the inner products. Now take the
infimum on the left over all of $T_{\sigma}(\mu,\xi)$, which includes the
compositions $\omega\circ\psi$ for all $\omega\in T_{\sigma}(\mu,\nu)$ and
$\psi\in T_{\sigma}(\nu,\psi)$, followed in turn by the infima over all
$\omega\in T_{\sigma}(\mu,\nu)$ and $\psi\in T_{\sigma}(\nu,\psi)$ on the right.
\end{proof}

Note that the middle term in the application of triangle inequality in
$\bigoplus_{l=1}^{n}H$ above is zero because we use the relative
tensor product $H=H_{\omega}\otimes_{\nu}H_{\psi}$ containing $H_{\varphi}$.
Exactly this point of the triangle inequality tends to run astray in other
transport plan based approaches to the Wasserstein distance. See in particular
\cite{dePT}, which in turn built on \cite{GMP}; also see \cite[Footnote
4]{GP2}.

The triangle inequality actually still works if one were to use all transport
plans, rather than just modular transport plans, a point which we come back to
in the final section.

\section{Symmetry}

We now proceed to prove that $W_{2}$ is symmetric. This is where the modular
property of our allowed transport plans (Definition \ref{modOord}) becomes important.

For clarity, in this section we initially work with two von Neumann algebras
$A$ and $B$ with faithful normal states $\mu$ and $\nu$ respectively, again
using the notation and conventions from Section \ref{AfdOordPlan}. After
setting up some general facts in this setting, we return to a single von
Neumann algebra when proving that $W_{2}$ is symmetric.

A suitable path to symmetry is provided by KMS-duals. Such duals and
KMS-symmetry were studied and applied in \cite{Pet, OPet, GL93, GL, C, FR,
DS2}. As in Section \ref{AfdOordPlan} we define $j_{\mu}:=J_{\mu}(\cdot
)^{\ast}J_{\mu}$ on $\mathcal{B}(G_{\mu})$ in terms of the modular conjugation
$J_{\mu}$ for $A$ associated to $\Lambda_{\mu}\in G_{\mu}$.

\begin{definition}
\label{KMS-duaal}Given a u.c.p. map $E:A\rightarrow B$ such that $\nu\circ
E=\mu$, we define its \emph{KMS-dual} (w.r.t. $\mu$ and $\nu$) as
\[
E^{\sigma}:=j_{\mu}\circ E^{\prime}\circ j_{\nu}:B\rightarrow A
\]
in terms of the dual $E^{\prime}:B^{\prime}\rightarrow A^{\prime}$ of $E$
defined by
\[
\left\langle \Lambda_{\mu},aE^{\prime}(b^{\prime})\Lambda_{\mu}\right\rangle
=\left\langle \Lambda_{\nu},E(a)b^{\prime}\Lambda_{\nu}\right\rangle
\]
for all $a\in A$ and $b^{\prime}=B^{\prime}$. (Regarding $E^{\prime}$, see
\cite{AC}, and \cite[Section 2]{DS2} for a summary).
\end{definition}

The basic properties of the KMS-dual needed here are set out in the next lemma.

\begin{lemma}
\label{KMS-dEiensk}In Definition \ref{KMS-duaal}, $E^{\sigma}$ is u.c.p. and
$\mu\circ E^{\sigma}=\nu$. Furthermore, $(E^{\sigma})^{\sigma}=E$. The map $E$
has a Hilbert space representation as a contraction $K:G_{\mu}\rightarrow
G_{\nu}$ defined through $Ka\Lambda_{\mu}=E(a)\Lambda_{\nu} $ for all $a\in
A$, and $E^{\sigma}$ is similarly represented by the contraction 
$J_\mu K^*J_\nu:G_\nu\rightarrow G_\mu$.
Consequently, the following equivalence holds:
\[
\mu(aE^{\sigma}(b))=\nu(E(a)b)
\]
for all $a\in A$ and $b\in B$, if and only if
\[
J_{\nu}K=KJ_{\mu}.
\]
\end{lemma}

\begin{proof}
Note that $E^{\prime}$ is u.c.p. and satisfies $\mu^{\prime}\circ E^{\prime
}=\nu^{\prime}$ according to \cite[Proposition 3.1]{AC}. The corresponding
properties for $E^{\sigma}$ then follow easily, while $(E^{\sigma})^{\sigma
}=E$ results from $(E^{\prime})^{\prime}=E$. That $K$ is well-defined, follows
from $\nu\circ E=\mu$ (which implies that $\nu(E(a)^{\ast}E(a))\leq
\nu(E(a^{\ast}a))=\mu(a^{\ast}a)$ by Kadison's inequality).
This implies that $E^{\prime}$ is represented by $K^{\ast}$, since
\[
\left\langle 
E'(b')\Lambda_{\mu},a\Lambda_{\mu}
\right\rangle
=
\left\langle 
K^*b'\Lambda_{\nu},a\Lambda_{\mu}
\right\rangle
\]
by the definition of $E^{\prime}$, which in turn means that $E^{\sigma}$ is
represented by $J_{\mu}K^{\ast}J_{\nu}$ by Definition \ref{KMS-duaal}.
The mentioned equivalence now follows by routine manipulations.
\end{proof}

This allows us to prove the following lemma in terms of the modular transport
plans from Definition \ref{modOord}, which subsequently leads to symmetry of
$W_{2}$.

\begin{lemma}
\label{algSim}For faithful normal states $\mu$ and $\nu$ on von Neumann
algebras $A$ and $B$ respectively, we have
\[
\left\|  \pi_{\mu}^{\omega}(a)\Omega_{\omega}-\pi_{\nu}^{\omega}%
(b)\Omega_{\omega}\right\|  _{\omega}=\left\|  \pi_{\nu}^{\omega^{\sigma}%
}(b)\Omega_{\omega^{\sigma}}-\pi_{\mu}^{\omega^{\sigma}}(a)\Omega
_{\omega^{\sigma}}\right\|  _{\omega^{\sigma}}%
\]
for every $\omega\in T_{\sigma}(\mu,\nu)$, where $\omega^{\sigma}\in
T_{\sigma}(\nu,\mu)$ is determined by $E_{\omega^{\sigma}}=E_{\omega}^{\sigma
}$. For emphasis we have written $\Omega_{\omega}$ for the cyclic vector
appearing in a cyclic representation associated to $\omega$, and similarly for
$\omega^{\sigma}$.
\end{lemma}

\begin{proof}
Since $\omega$ is modular, for the Hilbert space representation $K_{\omega} $
of $E_{\omega}$ we have $\Delta_{\nu}^{it}K_{\omega}=K_{\omega}\Delta_{\mu
}^{it}$ by \cite[Theorem 4.1]{DS2}, $\Delta_{\mu}$ and $\Delta_{\nu}$ being
the modular operators associated to $\Lambda_{\mu}$ and $\Lambda_{\nu}$
respectively.
Consequently $J_{\nu}K_{\omega}=K_{\omega}J_{\mu}$. From (\ref{normVsE}) and
Lemma \ref{KMS-dEiensk}, it then follows that%
\begin{align*}
\left\|  \pi_{\mu}^{\omega}(a)\Omega_{\omega}-\pi_{\nu}^{\omega}%
(b)\Omega_{\omega}\right\|  _{\omega}^{2} &  =\mu(a^{\ast}a)+\nu(b^{\ast
}b)-\nu(E_{\omega}(a)^{\ast}b)-\nu(b^{\ast}E_{\omega}(a))\\
&  =\nu(b^{\ast}b)+\mu(a^{\ast}a)-\mu(E_{\omega}^{\sigma}(b)^{\ast}
a)-\mu(a^{\ast}E_{\omega}^{\sigma}(b))\\
&  =\left\|  \pi_{\nu}^{\omega^{\sigma}}(b)\Omega_{\omega^{\sigma}}-\pi_{\mu
}^{\omega^{\sigma}}(a)\Omega_{\omega^{\sigma}}\right\|  _{\omega^{\sigma}}^{2}
\end{align*}
where 
$\omega^{\sigma}\in T(\nu,\mu)$ 
is determined by $E_{\omega^{\sigma}
}=E_{\omega}^{\sigma}$ according to \cite[Section 4]{DS2}. Note that
$\omega^{\sigma}\in T_{\sigma}(\nu,\mu)$ by \cite[Theorem 4.1]{DS2}, since
$(\sigma_{t}^{\mu})^{\sigma}
=\sigma_{-t}^{\mu}$, 
so 
$E_{\omega^{\sigma}}\circ\sigma_{t}^{\nu}
=(\sigma_{-t}^{\nu}\circ E_{\omega})^{\sigma}
=(E_{\omega}\circ\sigma_{-t}^{\mu})^{\sigma}
=\sigma_{t}^{\mu}\circ E_{\omega^{\sigma}}$.
\end{proof}

In particular we have symmetry of $W_{2}$:

\begin{proposition}
\label{sim}Consider any $k_{1},...,k_{n}\in A$, and let $W_{2}$ be the
associated Wasserstein distance on $\mathfrak{F}(A)$. Then%
\[
W_{2}(\mu,\nu)=W_{2}(\nu,\mu)
\]
for all $\mu,\nu\in\mathfrak{F}(A)$.
\end{proposition}

\begin{proof}
This follows from Definition \ref{W2} of $W_{2}$ and Lemma \ref{algSim}, since
for each $\omega$, every term in $I(\omega)$ (see (\ref{Iterme})) is equal to
the corresponding term in $I(\omega^{\sigma})$, while $(\omega^{\sigma
})^{\sigma}=\omega$ because of $(E_{\omega}^{\sigma})^{\sigma}=E_{\omega}$,
giving a one-to-one correspondence between $T_{\sigma}(\mu,\nu)$ and
$T_{\sigma}(\nu,\mu)$, which means we retain equality in the infima over
$T_{\sigma}(\mu,\nu)$ and $T_{\sigma}(\nu,\mu)$ respectively on the two
sides.
\end{proof}


\section{$W_{2}$ is a metric}\label{metriek}

So far in our proof that $W_{2}$ is a metric, we have not assumed that $A$ is
finitely generated. This will shortly become necessary. However, we postpone
making this assumption until it becomes essential.

One of the remaining properties of a metric is the following:

\begin{proposition}
\label{W2(mu,mu)=0}For $W_{2}$ in Definition \ref{W2}, we have%
\[
W_{2}(\mu,\mu)=0
\]
for all $\mu\in\mathfrak{F}(A)$.
\end{proposition}

\begin{proof}
This follows directly from (\ref{Ialt}) by setting $\omega=\delta_{\mu}\in
T_{\sigma}(\mu,\mu)$, i.e. $E_{\omega}=\operatorname*{id}_{A}$, which is
obviously a modular transport plan, since trivially $E_{\omega}\circ\sigma
_{t}^{\mu}=\sigma_{t}^{\mu}\circ E_{\omega}$.
\end{proof}

Aiming to achieve this simple result, was in fact the initial hint to approach
the cost $I$ heuristically through (\ref{heuristieseKoste}) and consequently
rigorously by Definition \ref{I}.

All that remains, is to show that $W_{2}(\mu,\nu)=0$ implies $\mu=\nu$. The
first of two lemmas towards this, is the following:

\begin{lemma}
\label{minOP}Consider any $k_{1},...,k_{n}\in A$, and let $W_{2}$ be the
associated Wasserstein distance on $\mathfrak{F}(A)$. For any $\mu,\nu
\in\mathfrak{F}(A)$, there then exists a modular transport plan $\omega\in
T_{\sigma}(\mu,\nu)$ such that
\[
W_{2}(\mu,\nu)=I(\omega)^{1/2}.
\]
\end{lemma}

\begin{proof}
It is a routine exercise to show that $T_{\sigma}(\mu,\nu)$ in Definition
\ref{modOord} is weakly* compact, since without loss we can view each element
of $T_{\sigma}(\mu,\nu)$ as a state on the maximal C*-tensor product
$A\otimes_{\text{max}}B^{\prime}$ (see for example \cite[Proposition 4.1]{D2}).
By the definition of $W_{2}$ there is a sequence $\omega_{q}\in T_{\sigma}
(\mu,\nu)$ such that $I(\omega_{q})^{1/2}\rightarrow W_{2}(\mu,\nu) $, which
then must have a weak* cluster point $\omega\in T_{\sigma}(\mu,\nu)$. The
lemma now follows by the following approximation:

Given $\varepsilon>0$, there is a $q_{0}$ such that
\[
\left|  I(\omega_{q})-W_{2}(\mu,\nu)^{2}\right|  <\varepsilon
\]
for all $q>q_{0}$. There exist $a_{1}^{\prime},...,a_{n}^{\prime}\in
A^{\prime}$ such that%
\[
\left\|  k_{l}\Lambda_{\nu}-a_{l}^{\prime\ast}\Lambda_{\nu}\right\|
<\varepsilon
\]
for $l=1,...,n$. Furthermore, there is a $q>q_{0}$ such that
\[
\left|  \omega_{q}(k_{l}\otimes a_{l}^{\prime})-\omega(k_{l}\otimes
a_{l}^{\prime})\right|  <\varepsilon
\]
for $l=1,...,n$. Using (\ref{Ialt}) we then find
\begin{align*}
\left|  I(\omega)-I(\omega_{q})\right|   &  \leq2\sum_{l=1}^{n}\left|
\nu(k_{l}^{\ast}E_{\omega_{q}}(k_{l}))-\nu(k_{l}^{\ast}E_{\omega}
(k_{l}))\right| \\
&  =2\sum_{l=1}^{n}\left|  \left\langle k_{l}\Lambda_{\nu},(E_{\omega_{q}
}(k_{l})-E_{\omega}(k_{l}))\Lambda_{\nu}\right\rangle \right| \\
&  \leq2\sum_{l=1}^{n}\left|  \left\langle k_{l}\Lambda_{\nu}-a_{l}
^{\prime\ast}\Lambda_{\nu},(E_{\omega_{q}}(k_{l})-E_{\omega}(k_{l}
))\Lambda_{\nu}\right\rangle \right| \\
&  +2\sum_{l=1}^{n}\left|  \left\langle \Lambda_{\nu},(E_{\omega_{q}}
(k_{l})-E_{\omega}(k_{l}))a_{l}^{\prime}\Lambda_{\nu}\right\rangle \right| \\
&  \leq4\sum_{l=1}^{n}\left\|  k_{l}\Lambda_{\nu}-a_{l}^{\prime\ast}
\Lambda_{\nu}\right\|  \left\|  k_{l}\right\| \\
&  +2\sum_{l=1}^{n}\left|  \omega_{q}(k_{l}\otimes a_{l}^{\prime}
)-\omega(k_{l}\otimes a_{l}^{\prime})\right| \\
&  <4\varepsilon\sum_{l=1}^{n}\left\|  k_{l}\right\|  +2n\varepsilon.
\end{align*}
Consequently,
\[
\left| I(\omega)-W_2(\mu,\nu)^2 \right|
<
4\varepsilon\sum_{l=1}^{n}\left\| k_l \right\|  
+2n\varepsilon+\varepsilon
\]
for all $\varepsilon>0$.
\end{proof}

This lemma is of some independent interest, since it says that the distance
$W_{2}(\mu,\nu)$ is always reached by some modular transport plan. To prove
that $W_{2}$ is a metric, it will be applied to the special case $W_{2}
(\mu,\nu)=0$.

The second lemma is where we see why $A$ is eventually taken to be finitely
generated, as is to be expected given the analogy to $\mathbb{R}^{n}$ in the Introduction.

\begin{lemma}
\label{EvsId}Consider any $k_{1},...,k_{n}\in A$ such that $\{k_{1}^{\ast
},...,k_{n}^{\ast}\}=\{k_{1},...,k_{n}\}$ and let $I$ be the associated
transport cost function. Let $R$ be the von Neumann subalgebra of $A$
generated by $\{k_{1},...,k_{n}\}$. Let $\mu,\nu\in\mathfrak{F}(A)$. If
$I(\omega)=0$ for some transport plan $\omega\in T(\mu,\nu)$, then the
restriction of $E_{\omega}$ to $R$ is the identity map:
\[
E_{\omega}|_{R}=\operatorname*{id}{}_{R}
\]
and consequently $\mu|_{R}=\nu|_{R}$.
\end{lemma}

\begin{proof}
Note that due to (\ref{normVsE}) and (\ref{Ialt})
\begin{align*}
&  \nu\left(  \left|  k_{l}-E_{\omega}(k_{l})\right|  ^{2}\right)  +\nu\left(
E_{\omega}(|k_{l}|^{2})-|E_{\omega}(k_{l})|^{2}\right) \\
&  =\mu(k_{l}^{\ast}k_{l})+\nu(k_{l}^{\ast}k_{l})-\nu(E_{\omega}(k_{l})^{\ast
}k_{l})-\nu(k_{l}^{\ast}E_{\omega}(k_{l}))\\
&  =0
\end{align*}
where $E_{\omega}(|k_{l}|^{2})-|E_{\omega}(k_{l})|^{2}\geq0$ by Kadison's
inequality, hence
\[
\nu\left(  \left|  k_{l}-E_{\omega}(k_{l})\right|  ^{2}\right)  =0\text{
\ \ and \ \ }\nu\left(  E_{\omega}(|k_{l}|^{2})-|E_{\omega}(k_{l}
)|^{2}\right)  =0
\]
for $l=1,...,n$. Since $\nu$ is faithful, the former implies%
\begin{equation}
E_{\omega}(k_{l})=k_{l}\label{E&id}
\end{equation}
while the latter implies
\[
E_{\omega}(k_{l}^{\ast}k_{l})=E_{\omega}(k_{l})^{\ast}E_{\omega}(k_{l})
\]
for $l=1,...,n$.

Setting
\[
A_{\omega}:= \left\{ a\in A:E_{\omega}(a^{\ast}a) =E_{\omega}(a)^{\ast
}E_{\omega}(a)\right\} ,
\]
it follows from \cite[Theorem 3.1]{Choi74} that $A_{\omega}$ is a (norm
closed) subalgebra of $A$ and that
\begin{equation}
A_{\omega}=\left\{  a\in A:E_{\omega}(ba)=E_{\omega}(b)E_{\omega}(a)\text{ for
all }b\in A\right\}  .\label{vermAlg}
\end{equation}
Since $\{k_{1}^{\ast},...,k_{n}^{\ast}\}=\{k_{1},...,k_{n}\}$, we have
$k_{1},...,k_{n},k_{1}^{\ast},...,k_{n}^{\ast}\in A_{\omega}$, 
so $A_{\omega}$ contains the $\ast$-algebra $R_{0}$ generated by 
$\{1,k_{1},...,k_{n}\}$.
Moreover, since $E_{\omega}$ is positive and therefore preserves the
involution, we see from (\ref{vermAlg}) that $E_{\omega}|_{R_{0}}$ is a unital
$\ast$-homomorphism, thus $E_{\omega}|_{R_{0}}=\operatorname*{id}_{R_{0}}$
because of (\ref{E&id}). As $E_{\omega}$ is normal, it follows that
$E_{\omega}|_{R}=\operatorname*{id}_{R}$, which implies that $\mu(r)=\nu\circ
E_{\omega}(r)=\nu(r)$ for all $r\in R$.
\end{proof}

The following corollary of this lemma is what we need to complete the proof
that $W_{2}$ is a metric:

\begin{corollary}
\label{mu=nu}If $A$ is generated by $\{k_{1},...,k_{n}\}$ in Lemma \ref{EvsId}
and $I(\omega)=0$ for some $\omega\in T(\mu,\nu)$, then $E_{\omega
}=\operatorname*{id}_{A}$, $\omega=\delta_{\mu}$ and $\mu=\nu$.
\end{corollary}

\begin{proof}
This follows from Lemma \ref{EvsId}, since $R=A$, and $\omega=\delta_{\nu
}\circ\left(  E_{\omega}\odot\text{id}_{A^{\prime}}\right)  =\delta_{\mu}$.
\end{proof}

In particular, this tells us that $\delta_{\mu}$ is the unique transport plan
from $\mu$ to itself attaining zero transport cost, but only the $\mu=\nu$
result is needed next.

Finally, we reach this paper's main result:

\begin{theorem}
\label{hoofSt}Let $A$ be a $\sigma$-finite von Neumann algebra. Assume that
$A$ is generated by $k_{1},...,k_{n}\in A$ such that $\{k_{1}^{\ast}
,...,k_{n}^{\ast}\}=\{k_{1},...,k_{n}\}$. Let $W_{2}$ be the Wasserstein
distance on $\mathfrak{F}(A)$ associated to $k_{1},...,k_{n}$ in Definition
\ref{W2}. Then $W_{2}$ is a metric.
\end{theorem}

\begin{proof}
By its definition, $W_{2}$ is real-valued and never negative. According to
Propositions \ref{driehoek} and \ref{sim} we know that $W_{2}$ satisfies the
triangle inequality and is symmetric. Proposition \ref{W2(mu,mu)=0} tells us
that $W_{2}(\mu,\mu)=0$. If $W_{2}(\mu,\nu)=0$, it follows from Lemma
\ref{minOP} followed by Corollary \ref{mu=nu} that $\mu=\nu$.
\end{proof}

If $A$ is not generated by $\{k_{1},...,k_{n}\}$ in Lemma \ref{EvsId}, then
$\mu\neq\nu$ is possible, as we now show, in which case $W_{2}$ is then only a pseudometric.

Suppose we have a von Neumann subalgebra $F$ of $A$ such that $\sigma_{t}
^{\mu}(F)=F$ for all $t$ and $R\subset F$. Then, by \cite{T72}, there is a
unique conditional expectation $E$ from $A$ onto $F$ such that $\mu\circ
E=\mu$. Assuming that
\[
\mu|_{F}=\nu|_{F},
\]
this gives a transport plan $\omega=\delta_{\nu}\circ(E\odot$id$_{A^{\prime}
})$ from $\mu$ to $\nu$. (This $\omega$ is an example of a relatively
independent coupling, which comes up in the theory of noncommutative joinings;
see \cite[Section 3]{D3}.) Then $I(\omega)=0$ by (\ref{Ialt}), since
$E_{\omega}=E$ restricted to $F$ is the identity map. In particular $W_{2}
(\mu,\nu)=0$.

In this setup there are cases with $\mu\neq\nu$, for example: Consider
$A=M\bar{\otimes}N$ for von Neumann algebras $M$ and $N$, and set $\mu
=\lambda\bar{\otimes}\zeta$ and $\nu=\lambda\bar{\otimes}\eta$ for faithful
normal states $\lambda$ on $M$, and $\zeta$ and $\eta$ on $N$. Taking
$F=M\otimes1$ and $k_{1},...,k_{n}\in F$, we satisfy all the requirements of
the previous paragraph, but $\mu\neq\nu$ if $\zeta\neq\eta$.

\section{Expanding the setting}

To conclude the paper, we briefly outline two ways of expanding the setting
above, to obtain somewhat weaker results, which emphasize the role our
assumptions play in proving that $W_{2}$ is a metric, and as an indication of
possible further avenues to explore. This is followed by questions regarding
further generalization.

First, we can allow all transport plans in Definition \ref{W2} of the
distance. Given $k_{1},...,k_{n}\in A$, we define an associated distance
function $d$\ on $\mathfrak{F}(A)$ by
\[
d(\mu,\nu):=\inf_{\omega\in T(\mu,\nu)}I(\omega)^{1/2}
\]
in terms of Definition \ref{I}.

With minor modifications to the foregoing work, one then obtains the following
variation on our main result:

\begin{proposition}
Let $A$ be a $\sigma$-finite von Neumann algebra. Assume that $A$ is generated
by $k_{1},...,k_{n}\in A$ such that $\{k_{1}^{\ast},...,k_{n}^{\ast}
\}=\{k_{1},...,k_{n}\}$. Let $d$ be the distance function above associated to
$k_{1},...,k_{n}$. Then $d$ is an asymmetric metric, that is to say, $d$
satisfies all the requirements of a metric, except that it may not be symmetric.
\end{proposition}

That $d$ may not be symmetric, could be very natural, as we have a direction
of transport involved, which now appears to be reflected in the distance function.

On the other hand, we have not shown that the modular property of transport
plans is the weakest condition ensuring symmetry, nor have we explicitly shown
$d$ not to be symmetric. The modular condition is very natural, but the
possibility of symmetry more generally is an open question.

The second variation is to drop the assumption that the states we work with
are faithful. That is to say, we define a distance function on the set
$\mathfrak{S}(M)$ of all normal states on a von Neumann algebra $M$. To do
this, let $p_{\zeta}$ be the support projection of $\zeta\in\mathfrak{S}(M)$.
Restrict to the von Neumann algebra $A=p_{\zeta}Mp_{\zeta}$ and replace
$\zeta$ by its restriction $\mu$ to $A$, which is a faithful normal state.
Similarly for $\eta\in\mathfrak{S}(M)$ to obtain the von Neumann algebra $B$
and faithful normal state $\nu$.

For $k_{1},...,k_{n}\in M$, we consider
\[
k_{l}^{\zeta}:=p_{\zeta}k_{l}p_{\zeta}\in A
\]
and write $k^{\zeta}=(k_{1}^{\zeta},...,k_{n}^{\zeta})$. Similarly for $\eta$.
Then apply the representation machinery from Section \ref{AfdOordPlan} to
$(A,\mu)$ and $(B,\nu)$, strictly speaking after setting up a cyclic
representation for both, to define
\[
\mathcal{I}(\omega)^{1/2}=\left\|  \pi_{\mu}^{\omega}(k^{\zeta})\Omega
-\pi_{\nu}^{\omega}(k^{\eta})\Omega\right\|  _{\oplus\omega}
\]
for $\omega\in T_{\sigma}(\mu,\nu)$, in place of Definition \ref{I}.

Now we define a distance function $\rho$\ on $\mathfrak{S}(M)$ associated to
$k_{1},...,k_{n}$, by
\[
\rho(\zeta,\eta)=\inf_{\omega\in T_{\sigma}(\mu,\nu)}\mathcal{I}(\omega)^{1/2}
\]
which leads to the next result, again by minor modifications to our previous
work:

\begin{proposition}
Let $M$ be a von Neumann algebra and consider any $k_{1},...,k_{n}\in M$. Let
$\rho$ be the distance function above associated to $k_{1},...,k_{n}$. Then
$\rho$ is an pseudometric on $\mathfrak{S}(M)$, that is to say, $\rho$
satisfies all the requirements of a metric, except that we may have
$\rho(\zeta,\eta)=0$ with $\zeta\neq\eta$.
\end{proposition}

There are some natural further questions:

We focussed on a finite set $\{k_{1},...,k_{n}\}$ of bounded operators to
define cost, in analogy to the coordinate functions on a bounded closed set in
$\mathbb{R}^{n}$. Can one expand on this and use appropriate infinite sets of
$k_{l}$ as well? Or can the $k_{l}$ be unbounded, but affiliated to the von
Neumann algebra in question?

Lastly, what would be the best way of adapting the approach of this paper to Wasserstein metrics of order other than 2?

\end{document}